\documentclass[11pt]{article}

\usepackage[T1]{fontenc}
\usepackage{lmodern}
\usepackage[margin=1.15in]{geometry}
\usepackage{amsmath,amssymb,amsthm,mathtools}
\usepackage{enumitem}
\usepackage{xcolor}
\usepackage[colorlinks=true,linkcolor=blue!55!black,citecolor=blue!55!black,
            urlcolor=blue!55!black]{hyperref}
\usepackage[nameinlink,capitalize,noabbrev]{cleveref}

\newtheorem{theorem}{Theorem}[section]
\newtheorem{proposition}[theorem]{Proposition}
\newtheorem{lemma}[theorem]{Lemma}

\theoremstyle{definition}
\newtheorem{definition}[theorem]{Definition}

\theoremstyle{plain}
\newtheorem{mainresult}{Theorem}
\newtheorem{mainaddendum}[mainresult]{Addendum}
\newtheorem{maincorollary}[mainresult]{Corollary}
\theoremstyle{plain}
\newtheorem{mainremark}[mainresult]{Remark}
\theoremstyle{plain}
\crefname{mainresult}{theorem}{theorems}
\Crefname{mainresult}{Theorem}{Theorems}
\crefname{mainaddendum}{addendum}{addenda}
\Crefname{mainaddendum}{Addendum}{Addenda}
\crefname{maincorollary}{corollary}{corollaries}
\Crefname{maincorollary}{Corollary}{Corollaries}
\crefname{mainremark}{remark}{remarks}
\Crefname{mainremark}{Remark}{Remarks}
\numberwithin{equation}{section}

\newcommand{\R}{\mathbb R}
\newcommand{\Wu}{W^{u}}
\newcommand{\Ws}{W^{s}}
\newcommand{\Wou}{W^{0u}}
\newcommand{\Wos}{W^{0s}}
\newcommand{\Wcu}{W^{cu}}
\newcommand{\Wcs}{W^{cs}}
\newcommand{\ds}{d^{s}}
\newcommand{\du}{d^{u}}

\title{Rigidity from the strong unstable foliation for higher-dimensional Anosov and partially hyperbolic flows}
\author{Andrey Gogolev}
\date{\today}

\begin{document}
\maketitle
\abstract{Recently Burniol Clotet \cite{Burniol} proved that two transitive Anosov flows on
3-manifolds with equivalent strong unstable
foliations must be conjugate up to time rescaling. We further explore the mechanism of Burniol Clotet's argument and treat the case of higher-dimensional Anosov flows and their isometric
extensions by compact Lie groups. }

\section{Results}
\label{sec:results}

Let $M$ be a closed manifold and let $X^t\colon M\to M$ be a $C^r$, $r\ge 1$, Anosov flow. We will consider $X^t$ to be the ``model'' flow and another Anosov flow $Y^t$ related to the model $X^t$ in the following way.
We say that their
strong unstable foliations are \emph{equivalent} if there exists a
homeomorphism $\varphi\colon N\to M$ such that
\[
  \varphi\bigl(\Wu_Y(y)\bigr)=\Wu_X\bigl(\varphi(y)\bigr),
  \qquad y\in N.
\]
We call such a homeomorphism an \emph{unstable equivalence}. The question is then to recover a direct link between dynamics of $X^t$ and $Y^t$.

The genesis of this type of questions goes back at least to Ratner who proved that measurable isomorphisms between the horocycle flows of
finite-volume hyperbolic surfaces are algebraic~\cite{Ratner} and, hence, the surfaces must be isometric. This result was later generalized to higher dimensional hyperbolic manifolds by Flaminio and Spatzier~\cite{Flaminio, FlaminioSpatzier}. The more direct predecessor to the topological question we consider here is the work of Marcus~\cite{Marcus}, who showed that if $X^t$ and $Y^t$ are geodesic flows on negatively curved surfaces which admit an unstable equivalence then $Y^t$ is conjugate to a rescaled flow $X^{\lambda t}$, where $\lambda$ is a fixed positive constant. In the same paper he claimed the same theorem for 3-dimensional suspension flows and posed the general question. Being unaware of Marcus' work Abe also proved the same result~\cite{Abe}. The proofs have a lot in common with Marcus' being more geometric, he argues that the equivalence preserves weak unstable foliation by looking at unstable horocycles on the universal cover of the surface, and Abe's being more dynamical. Recently Burniol Clotet generalized the arguments and resolved the question of Marcus in the setting of transitive 3-dimensional Anosov flows~\cite{Burniol}. Here we built upon the arguments of Burniol Clotet and some ideas of~\cite{GogolevRodriguezHertz,GogolevRodriguezHertz25} and pursue the same question in higher dimensions.

\begin{mainresult}
\label{thm:general-anosov}
Let $X^t\colon M\to M$ and $Y^t\colon N\to N$ be transitive
Anosov flows.  Additionally assume that the strong unstable distribution
$E^u_X$ is $C^1$.  If their strong unstable foliations are equivalent, then
we have the following dichotomy:
\begin{enumerate}[label=\textup{(\roman*)}]
\item either there exist $\lambda>0$ and a homeomorphism $\psi\colon N\to M$ such that
  \[
    \psi\circ Y^t=X^{\lambda t}\circ\psi,
    \qquad t\in\R;
  \]
\item or there is a non-zero continuous $DX^t$-invariant subbundle
  $E\subset E^s_X$ which integrates to an $X^t$-invariant foliation
  $\mathcal F\subset\Ws_X$; moreover, $E\oplus E^u_X$ integrates to a
  $X^t$-invariant foliation subfoliated by both $\mathcal F$ and $\Wu_X$.
\end{enumerate}
%The second alternative is the stable/unstable dual of the jointly integrable subbundle appearing in the Subbundle Theorem of Gogolev--Rodriguez Hertz~\cite{GogolevRodriguezHertz}.
\end{mainresult}

Without the $C^1$ assumption on $E^u_X$, one still obtains the following result.
\begin{mainaddendum}
\label{add:nondegenerate}
Let $X^t\colon M\to M$ and $Y^t\colon N\to N$ be transitive
Anosov flows.  Suppose that $X^t$ is locally
$su$-non-degenerate.  If their strong unstable foliations are equivalent,
then there exist $\lambda>0$ and a homeomorphism $\psi\colon N\to M$ such that
\[
  \psi\circ Y^t=X^{\lambda t}\circ\psi,
  \qquad t\in\R.
\]
\end{mainaddendum}
Here local $su$-non-degeneracy roughly means that stable and unstable
directions do not integrate together on any local configuration; see
\Cref{def:local-su-nondegenerate} for the precise definition.

It is unknown to the author if the second alternative in~\Cref{thm:general-anosov} or the $su$-degenerate case of the addendum can be further promoted to joint integrability of strong distributions $E^s_X$ and $E^u_X$. Still it is possible under some additional assumptions.

\begin{mainresult}
\label{thm:codimension-one}
Let $X^t\colon M\to M$ and $Y^t\colon N\to N$ be transitive
Anosov flows.  Suppose that their strong stable foliations are one-dimensional.
If their strong unstable foliations are equivalent, then there exist
$\lambda>0$ and a homeomorphism $\psi\colon N\to M$ such that
\[
  \psi\circ Y^t=X^{\lambda t}\circ\psi,
  \qquad t\in\R.
\]
\end{mainresult}

\begin{maincorollary}
\label{cor:generic-conservative}
Assume that $M$ is equipped with a smooth volume and has dimension at least $4$.  Let
$\mathcal A(M)$ be the space of smooth volume-preserving Anosov
flows which have one-dimensional strong
unstable subbundle.  There exists a $C^1$-open and $C^\infty$-dense subset
$\mathcal V\subset\mathcal A(M)$ with the following property.
If $X^t\in\mathcal V$ and $Y^t\in\mathcal A(M)$ have equivalent
strong unstable foliations, then there exist $\lambda>0$ and a
$C^1$ diffeomorphism $\psi\colon M\to M$ such that
\[
  \psi\circ Y^t=X^{\lambda t}\circ\psi,
  \qquad t\in\R.
\]
\end{maincorollary}

This corollary is the unstable-foliation version of generic smooth rigidity for
codimension-one volume-preserving Anosov
flows~\cite{GogolevRodriguezHertz}.  We note that
both~\Cref{cor:generic-conservative} and~\Cref{thm:codimension-one} apply to suspensions of
codimension-one Anosov diffeomorphisms and to the new examples of
Fenley--Mann--Potrie~\cite{FMP}.

\begin{mainresult}
\label{thm:compact-extension}
Let $\bar X^t\colon B\to B$ be an Anosov flow and let
$\pi\colon M\to B$ be a principal fiber bundle whose fibers are given by a right action of a compact Lie
group $K$.  Suppose that $X^t\colon M\to M$ is an isometric extension of
$\bar X^t$.
Assume that the strong unstable foliation $\Wu_X$ is minimal and that $X^t$ is locally
$su$-non-degenerate.  Let
$Y^t\colon N\to N$ be a dynamically coherent
partially hyperbolic flow.
If the strong unstable foliations of $X^t$ and $Y^t$ are equivalent,
then there exist $\lambda>0$, a continuous homomorphism $k\colon\R\to K$, and a homeomorphism
$\psi\colon N\to M$ such that
\[
  \psi\circ Y^t=R^t\circ X^{\lambda t}\circ\psi,
  \qquad t\in\R,
\]
where $R^t(y)=y\cdot k(t)$, $y\in M$.
\end{mainresult}

In general, the drift along the compact group in
\Cref{thm:compact-extension} cannot be removed.  Indeed, for every
continuous homomorphism $k\colon\R\to K$, the flows $X^t$ and
$R^t\circ X^t$ have the same
strong unstable foliation.

Combining our results with marked-length-spectrum rigidity
\cite{Hamenstadt,GuillarmouLefeuvre} also yields the following geometric
corollaries.
\begin{maincorollary}
\label{cor:geodesic-rigidity}
Let $(M,g_0)$ and $(N,g)$ be negatively curved
$n$-manifolds, $n\geq3$, and assume that $(M,g_0)$ is locally symmetric.
Let $X^t\colon T^1M\to T^1M$ and $Y^t\colon T^1N\to T^1N$ be their
geodesic flows.
If their strong unstable foliations are equivalent, then there exist
 an isometry
\[
  F\colon (N,\lambda^2g)\longrightarrow(M,g_0),
\]
where $\lambda>0$ is a constant.
\end{maincorollary}

\begin{maincorollary}
\label{cor:geodesic-local-rigidity}
Let $g_0$ be a negatively curved metric on $M$, $\dim M=n\geq2$, and fix an
integer $k>3n/2+8$.  There exists $\varepsilon>0$ with the following
property.  Let $g$ be a negatively curved metric such that
\[
  \|g-g_0\|_{C^k}<\varepsilon.
\]
Suppose that the
strong unstable foliations of the geodesic flows of $g_0$ and $g$ are
equivalent via a homeomorphism homotopic to the identity.
Then there exist an isometry
\[
  F\colon(M,\lambda^2g)\longrightarrow(M,g_0),
\]
where $\lambda>0$ is a constant.
\end{maincorollary}

\begin{mainremark}\label{rem:frame-mls}
Let $(M_i,g_i)$, $i=1,2$, be negatively curved $n$-manifolds,
$n\geq3$, and let $X^t$ and $Y^t$ be their oriented frame flows,
respectively.  Assume that $\Wu_X$ is minimal (for example, when $n$ is odd, $n\neq7$,~\cite{BrinGromov}, and when $g_1$ has constant negative
curvature \cite{Brin}).  If their strong unstable foliations are equivalent,
then the geodesic flows on $(M_2,g_2)$ and $(M_1,g_1)$ are conjugate up to time rescaling.
%there exist $\lambda>0$ and an isomorphism
%$\rho\colon\pi_1(M_2)\to\pi_1(M_1)$ induced by a homotopy equivalence such that
%\[
 % \ell_{g_1}(\rho[\gamma])
 % =\lambda\,\ell_{g_2}([\gamma]),
  %\qquad [\gamma]\in[\pi_1(M_2)].
%\]
%Thus $(M_2,\lambda^2g_2)$ and $(M_1,g_1)$ have the same marked length spectrum.
\end{mainremark}

\subsection*{Comments on AI} Upon reading~\cite{Burniol} the author realized that the arguments could be generalized to the codimension one setting and also to yield the dichotomy of Theorem~\ref{thm:general-anosov}. At the time, in July 2026, the author only began experimenting with AI. Upon feeding two pages of notes outlining the proof to Claude Opus 4.8, it produced 31 page AI-slop ``paper'' proving Theorems 1 and~3. To the author's surprise (at the time), while terribly written the ``paper'' was mostly correct. After some expositional improvements by Claude Opus and then ChatGPT 5.6 Sol, the author had to largely rewrite the paper which included fixing some mistakes and a lot of streamlining of the arguments. Needless to say, the author bears full responsibility for correctness of this paper.

\subsection*{Acknowledgements} The author would like to thank Sergi Burniol Clotet and Federico Rodriguez Hertz for helpful feedback on the first draft.

\section{Preliminaries}
\label{sec:preliminaries}

We recall well-known definitions and results, and at the same time
establish some notation.
\begin{definition}
A flow $X^t$ on $M$ is \emph{Anosov} if  there is a continuous $DX^t$-invariant
splitting $ TM=E^s\oplus\R X\oplus E^u $
and constants $C,\chi>0$ such that, for $t\geq0$,
\[
  \|DX^t v^s\|\leq Ce^{-\chi t}\|v^s\|,
  \qquad
  \|DX^{-t}v^u\|\leq Ce^{-\chi t}\|v^u\|
\]
for $v^s\in E^s$ and $v^u\in E^u$.
\end{definition}

The strong stable and strong unstable foliations are denoted by $\Ws$ and $\Wu$.
The
weak stable and weak unstable foliations tangent to $E^s\oplus \R X $ and $\R X\oplus E^u$ are denoted by $\Wos$ and $\Wou$,
respectively. 
For $\sigma\in\{s,u,0s,0u\}$, the notation
$W^\sigma_{\mathrm{loc}}(x)$ means a fixed local plaque
through $x$, while $W^\sigma_\rho(x)$ denotes the intrinsic radius-$\rho$
plaque.  We write $\ds$ and $\du$ for the intrinsic distances on strong
stable and strong unstable leaves.
 When two
flows are compared, a subscript such as $X$ or $Y$ is used for clarity. 
\begin{definition}
The strong stable and strong unstable foliations of an Anosov flow are
\emph{jointly integrable} if $E^s\oplus E^u$ integrates to an
$X^t$-invariant foliation subfoliated by both $\Ws$ and $\Wu$.
\end{definition}

\begin{theorem}[Anosov alternative~\cite{Anosov, Plante}]\label{thm:anosov-alternative}
Let $X^t$ be a transitive Anosov flow.  Then either
all strong stable and strong unstable manifolds are dense and $X^t$ is
topologically mixing, or
$X^t$ is a suspension over an Anosov diffeomorphism by a constant roof
function.
\end{theorem}

\begin{theorem}[Plante~\cite{Plante}, Theorem 3.7]
\label{thm:plante-codimension-one}
Let $X^t$ be a transitive Anosov flow for which
$\dim E^s=1$ or $\dim E^u=1$.  If $\Ws$ and $\Wu$ are jointly
integrable, then every leaf of the integrating foliation is a compact $C^1$
global cross-section with constant return time. 
\end{theorem}

\begin{definition}
A flow $X^t$ on $M$ is \emph{partially hyperbolic}
if there is a continuous $DX^t$-invariant splitting
\[
  TM=E^s\oplus E^c\oplus E^u,
  \qquad \R X\subset E^c,
\]
such that $E^s$ and $E^u$ are uniformly contracted and expanded,
respectively, and the center has intermediate behavior. That is, there exist $C,\chi>0$ such that, for $t\geq0$,
\[
\begin{alignedat}{2}
 \|DX^t|_{E^s}\|&\leq Ce^{-\chi t},\qquad &
 \|DX^{-t}|_{E^u}\|&\leq Ce^{-\chi t},\\
 \frac{\|DX^t|_{E^s}\|}{m(DX^t|_{E^c})}
 &\leq Ce^{-\chi t},\qquad &
 \frac{\|DX^t|_{E^c}\|}{m(DX^t|_{E^u})}
 &\leq Ce^{-\chi t}.
\end{alignedat}
\]
where $m(A)=\|A^{-1}\|^{-1}$.  The flow is \emph{dynamically coherent} if
$E^s\oplus E^c$ and $E^c\oplus E^u$ integrate to invariant
foliations $\Wcs$ and $\Wcu$. In this case we let $W^c=\Wcs\cap\Wcu$.
\end{definition}

For a partially hyperbolic flow, we also write $\Ws$ and $\Wu$ for the
strong stable and strong unstable foliations tangent to $E^s$ and $E^u$,
respectively. For $\sigma\in\{s,u,cs,cu,c\}$, we use the same local plaque
notation as in Anosov case.

\begin{definition}\label{def:compact-group-extension}
Let $\bar X^t\colon B\to B$ be an Anosov flow and let
$\pi\colon M\to B$ be a principal fiber bundle whose fiber is a compact Lie
group $K$.  We fix a $K$-invariant metric on $M$.  An \emph{isometric extension} of $\bar X^t$ is a flow
$X^t$ on $M$ such that
\[
  \pi\circ X^t=\bar X^t\circ\pi,
  \qquad
  X^t\circ R_k=R_k\circ X^t,
  \qquad \forall k\in K,\ t\in\R,
\]
where $R_k(y)=y\cdot k$ is the principal right action.  Then $X^t$ is
partially hyperbolic with
$
  E^c=\R X\oplus\ker D\pi$.
\end{definition}

We use the following notation for the local product bracket.  For an Anosov
flow, whenever $x$ and $y$ are sufficiently close, we denote
\[
  [x,y]
  :=
  \Ws_{\mathrm{loc}}(x)\cap\Wou_{\mathrm{loc}}(y).
\]
For a dynamically coherent partially hyperbolic flow, we use the same
notation with $\Wcu_{\mathrm{loc}}(y)$ in place of
$\Wou_{\mathrm{loc}}(y)$.

\begin{definition}\label{def:temporal-distance}
For an Anosov flow $X^t$, let $w\in\Ws_{\mathrm{loc}}(x)$ and
$z\in\Wu_{\mathrm{loc}}(x)$.  The \emph{temporal distance}
$\Delta(w,z)$ is the unique small time such that
\[
  X^{\Delta(w,z)}([z,w])\in\Wu_{\mathrm{loc}}(w).
\]
\end{definition}

\begin{definition}\label{def:local-su-nondegenerate}
Let $X^t$ be a dynamically coherent partially hyperbolic flow.  We call
$X^t$ \emph{locally $su$-non-degenerate} if, whenever
$w\in\Ws_{\mathrm{loc}}(x)\setminus\{x\}$, the local strong-stable
holonomy from $\Wcu_{\mathrm{loc}}(x)$ to
$\Wcu_{\mathrm{loc}}(w)$ does not carry an open strong unstable plaque
through $x$ into the strong unstable plaque through $w$.
This definition also applies to Anosov flow, in this case it says that for every
such $w$ the function
\[
  z\longmapsto\Delta(w,z),
  \qquad z\in\Wu_{\mathrm{loc}}(x),
\]
is not identically zero on any neighborhood of $x$ in
$\Wu_{\mathrm{loc}}(x)$.  
\end{definition}

\begin{definition}\label{def:contact-flow}
Let $M$ have dimension $2n+1$.  A smooth flow $X^t$ on
$M$, with generator $X$, is a \emph{contact flow} if it is the Reeb flow of
a smooth contact form $\alpha$; that is,
\[
  \alpha\wedge(d\alpha)^n\neq0,\qquad
  \alpha(X)=1,\qquad
  \iota_X d\alpha=0.
\]
In particular, $\alpha$ is $X^t$-invariant.  If $X^t$ is also Anosov, we
call it a \emph{contact Anosov flow}.
\end{definition}

For a contact Anosov flow,
$E^s\oplus E^u=\ker\alpha$.  Moreover, $d\alpha$ vanishes on each of
$E^s$ and $E^u$ and pairs these two subbundles non-degenerately.

\begin{proposition}\label{prop:contact-su-nondegenerate}
Every smooth contact Anosov flow is locally $su$-non-degenerate.  More
precisely, if $w\in\Ws_{\mathrm{loc}}(x)\setminus\{x\}$, then there are
$c>0$ and $z_n\in\Wu_{\mathrm{loc}}(x)$ with $z_n\to x$ such that
\[
  |\Delta(w,z_n)|\geq c\du(x,z_n)^2.
\]
Moreover,
if $X^t$ is an isometric extension of a smooth contact Anosov flow
$\bar X^t$, then $X^t$ is locally $su$-non-degenerate as well.
\end{proposition}

\begin{proof}
Let $x_n=X^n(x)$, $w_n=X^n(w)$, and
$a_n=\ds(x_n,w_n)\to 0$, $n\to\infty$.  In stable leafwise exponential
coordinates write $w_n=\exp^s_{x_n}(\bar w_n)$.  Since $d\alpha$ pairs
$E^s$ and $E^u$ non-degenerately, we can choose
$\bar v_n\in E^u(x_n)$ such that
\[
  \|\bar v_n\|=\|\bar w_n\|=a_n,
  \qquad
  |d\alpha(\bar v_n,\bar w_n)|\geq c_0a_n^2.
\]
Let $z'_n=\exp^u_{x_n}(\bar v_n)$.  Then Liverani's temporal-distance
estimate~\cite[Lemma~B.7]{Liverani} gives a smaller constant $c_1>0$ such that
\[
  |\Delta(w_n,z'_n)|\geq c_1a_n^2
\]
Let $z_n=X^{-n}(z'_n)$.  Iterating backwards we have $\du(x,z_n)\leq Ca_n$, while flow
invariance of temporal distance gives
$\Delta(w,z_n)=\Delta(w_n,z'_n)$.  The asserted estimate follows, and in
particular the temporal-distance function cannot vanish on a neighborhood
of $x$ in $\Wu_{\mathrm{loc}}(x)$.

For an isometric extension, the projection $\pi$
maps local strong stable and strong unstable plaques of $X^t$ onto the
corresponding plaques of $\bar X^t$ and intertwines their holonomies.
If a local stable holonomy upstairs carried an open strong unstable plaque
into a strong unstable plaque, its projection would do the same downstairs,
a contradiction. 
\end{proof}

\section{Anosov flows with $C^1$ unstable foliation}
\label{sec:holonomy-subbundles}

In this section we describe a general construction for an Anosov flow $X^t$ with $C^1$ regular unstable distribution $E^u$. It is well-known that if $E^u$ is $C^1$ then the unstable foliation $W^u$ has $C^1$ holonomies. The author believes that this discussion has some independent interest. It could also be carried out for partially hyperbolic systems. 

Recall that $E^{0s}=E^s\oplus\R X$.  For $z\in\Wu(y)$, let
\[
 \operatorname{Hol}^u_{y,z}\colon
 \Wos_{loc}(y)\longrightarrow \Wos_{loc}(z)
\]
be the strong unstable holonomy between weak stable leaves.
Unstable leaves are simply connected, so this holonomy is independent of the
choice of path from $y$ to $z$.  It is a $C^1$ diffeomorphism between
sufficiently small neighborhoods of its basepoints.  The differential
$D_y\operatorname{Hol}^u_{y,z}$ identifies $E^{0s}(y)$
with $E^{0s}(z)$.

For $\rho>0$, let
\begin{equation}
 E_\rho(y)=\bigcap_{z\in\Wu_\rho(y)}
 \bigl(D_y\operatorname{Hol}^u_{y,z}\bigr)^{-1}E^s(z).
\end{equation}
Define
\begin{equation}
 E_{\min}(y)=\lim_{\rho\searrow0}E_\rho(y),\qquad
 E_{\max}(y)=\lim_{\rho\to\infty}E_\rho(y)
 =\bigcap_{z\in\Wu(y)}
 \bigl(D_y\operatorname{Hol}^u_{y,z}\bigr)^{-1}E^s(z).
\end{equation}
We have
\[
 E_{\max}(y)\subset E_\rho(y)\subset E_{\min}(y)\subset E^s(y).
\]

The following lemma is immediate from the definitions.

\begin{lemma}\label{lem:holonomy-invariance}
Both $E_{\min}$ and $E_{\max}$ are $DX^t$-invariant.  Moreover,
$E_{\max}$ is invariant under unstable holonomy:
\begin{equation}
 D_y\operatorname{Hol}^u_{y,z}\bigl(E_{\max}(y)\bigr)
 =E_{\max}(z),\qquad z\in\Wu(y).
\end{equation}
\end{lemma}

\begin{lemma}\label{lem:holonomy-constant-rank}
The function $d(y)=\dim E_{\max}(y)$ is upper semicontinuous.
If $\Wu$ is minimal, then $d$ is constant and there exists $L>0$ such that
\begin{equation}
 E_{\max}(y)=E_L(y)\qquad\text{for every }y\in M.
\end{equation}
\end{lemma}

\begin{proof}
Since $E_{\max}(y)$ is finite dimensional, we can choose finitely many
$z_i\in\Wu(y)$ whose pulled-back stable subspaces intersect in
$E_{\max}(y)$.  We choose such a collection of minimal cardinality, then the
corresponding normal vectors are linearly independent.  The corresponding holonomies are $C^1$ maps
$h_i=\operatorname{Hol}^u_{y,z_i}$ which we can restrict to a common weak stable disk
$D^{0s}$ about $y$.
For any $w\in D^{0s}$, the same holonomy map, can be regarded as
unstable holonomy from $w$ to $h_i(w)$.  Consequently,
\begin{equation}\label{eq:holonomy-finite-constraints}
 E_{\max}(w)\subset
 K(w):=\bigcap_i(D_wh_i)^{-1}E^s(h_i(w)).
\end{equation}
The subspaces in this intersection vary continuously with $w$.
After shrinking $D^{0s}$ if needed, their normal vectors remain independent, so
$\dim K(w)=\dim E_{\max}(y)=d(y)$.  Hence $d(w)\leq d(y)$ on $D^{0s}$.
This proves upper semicontinuity on weak stable leaves.  Unstable
holonomy invariance then gives upper semicontinuity on $M$.

Suppose now that $\Wu$ is minimal.  The minimum locus of $d$ is
nonempty, open, and saturated by entire unstable leaves.  Its complement
is closed and unstable-saturated, so minimality makes that complement
empty.  Thus $d$ is constant, and the inclusion in
\eqref{eq:holonomy-finite-constraints} is now an equality.

The holonomies $h_i$ persist on $D^{0s}$, and
their unstable path lengths are bounded above by $L_y$.  Hence we have 
\begin{equation}\label{eq:holonomy-uniform-length}
 E_{\max}(w)=E_{L_y}(w)(y)\qquad\text{for every }w\in D^{0s}.
\end{equation}
This is true for any $y$. Hence, by compactness we have a finite cover and the corresponding maximal value $L$ which gives~\eqref{eq:holonomy-uniform-length}.
\end{proof}

\begin{lemma}\label{lem:holonomy-extremal-equality}
If $\Wu$ is minimal, then
\[
 E_{\min}(y)=E_{\max}(y)\qquad\text{for every }y\in M.
\]
In fact, $E_\rho(y)=E_{\max}(y)$ for every $\rho>0$.
\end{lemma}

\begin{proof}
By stabilization, $E_{\min}(y)=E_\rho(y)$ for some $\rho>0$.
Let $L$ be given by \Cref{lem:holonomy-constant-rank}.  For sufficiently
large $t>0$, backward contraction along unstable leaves gives
\[
 X^{-t}\bigl(\Wu_L(X^t y)\bigr)\subset\Wu_\rho(y).
\]
Hence,
\[
 DX^tE_\rho(y)\subset E_L(X^t y)=E_{\max}(X^t y).
\]
and by flow invariance we conclude that $E_{\min}(y)=E_\rho(y)\subset E_{\max}(y)$, hence they are equal.
\end{proof}

\begin{proposition}\label{prop:holonomy-joint-foliation}
Assume that $\Wu$ is minimal.  The common subbundle
$E=E_{\min}=E_{\max}\subset E^s$ is continuous and $DX^t$-invariant.
It integrates to an $X^t$-invariant foliation $\mathcal F\subset\Ws$
whose restriction to every stable leaf is a $C^1$ foliation.
Moreover, $E\oplus E^u$ integrates to an $X^t$-invariant $C^1$
foliation subfoliated by $\mathcal F$ and $\Wu$.
\end{proposition}

\begin{proof}
The holonomies $h_i$ from the proof of
\Cref{lem:holonomy-constant-rank} provide local $C^1$ foliation charts
tangent to $E$ on weak stable disks.  Thus
$E$ is continuous and integrates to a foliation $\mathcal F\subset\Ws$.
Because the maps $h_i$ which locally define $\mathcal F$~(3.5) are unstable holonomies, $\mathcal F$ integrates jointly with $W^u$.  
\end{proof}

\section{Proofs}
\label{sec:proofs}

In this section we prove \Cref{thm:general-anosov} and \Cref{add:nondegenerate} and then
deduce \Cref{thm:codimension-one} and \Cref{cor:generic-conservative}.  We
would like to emphasize that our proof is merely a crystallization of
Burniol Clotet's proof~\cite{Burniol}, which, in turn, is based on Abe's~\cite{Abe} and Marcus'~\cite{Marcus}: we replace the algebraic renormalization argument with
a soft dynamical one and push the argument into higher dimensions.  We first
treat the case in which the unstable equivalence does not respect the weak unstable
foliation.  Renormalizing a non-trivial stable displacement produces a
non-trivial local $su$-integrability relation.  For the dichotomy, the
assumption that $E^u$ is $C^1$ makes the
associated temporal-distance functions $C^1$ along stable leaves, allowing us to create an
integrable invariant subbundle.  We then treat, in
\Cref{prop:general-centered}, the case in which the equivalence respects the
weak unstable foliation.  There the transported flow preserves each
weak unstable leaf of $X^t$, and correcting maps produce the conjugacy as in~\cite{Burniol, Abe}.

Note that we can dispose of the jointly integrable case.  Indeed, the
foliation tangent to $E^s\oplus E^u$ is subfoliated by $\Ws$ and
$\Wu$, so the second alternative of \Cref{thm:general-anosov} holds.  Joint integrability is also impossible under the hypothesis of~\Cref{add:nondegenerate}.

Thus we consider the case in which $\Ws$ and $\Wu$ are not jointly
integrable.  The suspension case in \Cref{thm:anosov-alternative} is then
ruled out, so $X^t$ is topologically mixing and $\Wu$ is minimal; we will use
this repeatedly.

\subsection{The local displacement}

Fix an unstable equivalence $\varphi\colon N\to M$.  It is convenient to
transplant the flow $Y^t$ to $M$.  Define
\[
  \widehat Y^t=\varphi\circ Y^t\circ\varphi^{-1}.
\]
Thus $\widehat Y^t$ is a continuous flow on $M$ and it preserves the strong
unstable foliation of $X^t$:
\[
  \widehat Y^t\bigl(\Wu(y)\bigr)
  =\Wu\bigl(\widehat Y^t(y)\bigr).
\]

\begin{lemma}\label{lem:general-local-decomposition}
For a sufficiently small $\delta>0$ there exist continuous maps
\[
  p,q\colon M\times(-\delta,\delta)\longrightarrow M
  \quad\text{and}\quad
  \beta\colon M\times(-\delta,\delta)\longrightarrow\R
\]
such that
\begin{align}
  p(y,t)&=[y,\widehat Y^t(y)]
  =\Ws_{\mathrm{loc}}(y)
   \cap\Wou_{\mathrm{loc}}\bigl(\widehat Y^t(y)\bigr),
   \\
  q(y,t)&=X^{\beta(y,t)}(p(y,t))
  \in\Wu(\widehat Y^t(y)).\label{eq:general-q}
\end{align}
Moreover,
\begin{equation}\label{eq:general-small}
 \sup_{y\in M}\bigl(\ds(y,p(y,t))+|\beta(y,t)|\bigr)
 \longrightarrow0
 \qquad\text{as }t\longrightarrow0,
\end{equation}
and
\begin{equation}\label{eq:general-u-bound}
  \du\bigl(q(y,t),\widehat Y^t(y)\bigr)\leq C_0
\end{equation}
for a uniform constant $C_0$.  For fixed $t$ the restriction
\[
  q(\,\mathord{\cdot}\,,t)|_{\Wu(y)}\colon
  \Wu(y)\longrightarrow\Wu(\widehat Y^t(y))
\]
is a homeomorphism.
\end{lemma}

\begin{proof}
For sufficiently small $t$, $p$, $q$ and $\beta$ are all uniquely defined by local product
structure. Uniform local product structure gives continuity and
all listed properties.
\end{proof}

Fix $\delta_0>0$ so that the local decomposition in
\Cref{lem:general-local-decomposition} is defined for $|t|<\delta_0$.
Call such $t$ \emph{centered} if
\begin{equation}
  p(y,t)=
  \Ws_{\mathrm{loc}}(y)
  \cap\Wou_{\mathrm{loc}}\bigl(\widehat Y^t(y)\bigr)=y
  \qquad\text{for every }y\in M.
\end{equation}

\begin{lemma}\label{lem:centered-times}
If $|t|<\delta_0$ and $p(y,t)=y$ for some $y\in M$, then $t$ is centered.
\end{lemma}

\begin{proof}
The set $S_t=\{y:p(y,t)=y\}$ is closed.  Local uniqueness of the
intersection defining $p$, together with the fact that $\widehat Y^t$
preserves $\Wu$, shows that $S_t$ is uniformly locally $\Wu$-saturated.  Hence it is
$\Wu$-saturated.  By minimality, if $S_t$ is nonempty, then $S_t=M$.
\end{proof}

{
\begin{lemma}\label{lem:centered-accumulation}
If there are non-zero centered times arbitrarily close to zero, then every
$t$ with $|t|<\delta_0$ is centered.  Moreover, $\varphi$ maps the
weak unstable foliation of $Y^t$ onto that of $X^t$.
\end{lemma}

\begin{proof}
Centeredness is equivalent to
$\widehat Y^t(y)\in\Wou(y)$ for every $y$, and the set of centered times is
closed by continuity.  If $t_1$ and $t_2$ are centered and
$|t_1+t_2|<\delta_0$, then
$\widehat Y^{t_1+t_2}(y)\in\Wou(y)$ for every $y$, so $t_1+t_2$ is
centered.  Let $t_n\to0$ be non-zero centered times.
For any $|t|<\delta_0$, choose integers $k_n$ such that $k_nt_n\to t$. Since $k_nt_n$ is centered, so is $t$.
 Subdividing an arbitrary
time now gives $\widehat Y^t(y)\in\Wou_X(y)$ for every $t\in\R$.  Since
$\varphi$ maps $\Wu_Y$-leaves onto $\Wu_X$-leaves, this gives
$\varphi(\Wou_Y(x))=\Wou_X(\varphi(x))$.
\end{proof}
}

\subsection{Renormalization of a non-centered time}

Given a small unstable plaque $D^u(y)\subset\Wu(y)$, denote its local stable
saturation by
\[
  \operatorname{Sat}^s_{\mathrm{loc}}(D^u(y))
  =\bigcup_{z\in D^u(y)}\Ws_{\mathrm{loc}}(z).
\]
For $y'\in\Ws_{\mathrm{loc}}(y)$ write $y'\in\mathcal I(y)$ if for a sufficiently small unstable plaque
\begin{equation}
  D^u(y')
  \subset
  \operatorname{Sat}^s_{\mathrm{loc}}
  \bigl(D^u(y)\bigr).
\end{equation}
This is a local joint integrability condition.
After shrinking plaques approprietly, the
same condition holds with $y$ and $y'$ interchanged.

\begin{lemma}\label{lem:off-diagonal-su}
If there are non-centered times arbitrarily close to 0, then there exist
$y\in M$ and $y'\in\mathcal I(y)\setminus\{y\}$.
\end{lemma}

\begin{proof}
{ Recall that by \Cref{lem:centered-accumulation}, every sufficiently small non-zero time
is non-centered.  Choose $t_n\searrow0$.
}
Fix $y_0\in M$ and let
\[
  p_n=p(y_0,t_n)\in\Ws_{\mathrm{loc}}(y_0).
\]
Then $p_n\neq y_0$.  Fix a small
$L>0$.  Flowing backwards expands intrinsic stable distance $d^s$, so there is a
first time $T_n>0$ such that
\begin{equation}\label{eq:general-stopping}
 \ds\bigl(X^{-T_n}(y_0),X^{-T_n}(p_n)\bigr)=L.
\end{equation}
Equation \eqref{eq:general-small} implies $T_n\to\infty$.  Passing to a
subsequence, we have
\[
  y_n:=X^{-T_n}(y_0)\longrightarrow y,
  \qquad
  y_n':=X^{-T_n}(p_n)\longrightarrow y'.
\]
By \eqref{eq:general-stopping},
\[
  y'\in\Ws_{\mathrm{loc}}(y),
  \qquad \ds(y,y')=L.
\]

For $w\in\Wu(y_n)$ set $z=X^{T_n}(w)\in\Wu(y_0)$ and define
\begin{align*}
 F_n(w)&=X^{-T_n}\bigl(\widehat Y^{t_n}(z)\bigr),\\
 P_n(w)&=X^{-T_n}(p(z,t_n)),\\
 Q_n(w)&=X^{-T_n}(q(z,t_n)).
\end{align*}
The map $F_n$ is a homeomorphism from $\Wu(y_n)$ onto an unstable leaf,
and $Q_n$ is the weak stable holonomy homeomorphism between the same two
leaves.  Backward contraction along $\Wu$ and
\eqref{eq:general-u-bound} yield
\begin{equation}\label{eq:general-FQ}
 \sup_{w\in\Wu(y_n)}
 \du\bigl(F_n(w),Q_n(w)\bigr)
\longrightarrow0,\, n\to\infty.
\end{equation}
Also, \eqref{eq:general-q} gives the exact flow displacement
\begin{equation}\label{eq:general-QP}
 Q_n(w)=X^{\beta(z,t_n)}(P_n(w)),
 \qquad
 \sup_{w\in\Wu(y_n)}
 d\bigl(Q_n(w),P_n(w)\bigr)\longrightarrow0, \quad n\to\infty,
\end{equation}
where the convergence follows from \eqref{eq:general-small}.

Choose a small $\rho>0$ and consider the unstable plaques
\[
  D_n^u:=\Wu_\rho\bigl(F_n(y_n)\bigr).
\]
Since $P_n(y_n)=y_n'$, equations
\eqref{eq:general-FQ}--\eqref{eq:general-QP} show that the centers
$F_n(y_n)$ and $y_n'$ converge together to $y'$.  Thus $D_n^u$ and the plaques
$\Wu_\rho(y_n')$ both converge in the $C^0$ plaque topology to
$\Wu_\rho(y')$.

The holonomies $Q_n$ are uniformly proper on $W^u_{loc}(y_n)$.  Hence, after shrinking $\rho$, there exists
$\rho'>\rho$ such that every sequence $\{z_n\}\in D_n^u$ can be written as
$z_n=F_n(w_n)$ with $w_n\in\Wu_{\rho'}(y_n)$.  By passing to a subsequence, we have
$w_n\to w\in\Wu_{\mathrm{loc}}(y)$.  By
\eqref{eq:general-FQ}--\eqref{eq:general-QP}, the points $z_n$ and
$P_n(w_n)$ have the same limit.  But
\[
  P_n\bigl(\Wu_{\rho'}(y_n)\bigr)
  \subset\operatorname{Sat}^s_{\mathrm{loc}}\bigl(\Wu_{\rho'}(y_n)\bigr),
\]
so continuity of the stable foliation places this limit in
$\Ws_{\mathrm{loc}}(w)$.  Since $\{z_n\}$ was arbitrary, it follows that
\[
  \Wu_{\rho}(y')
  \subset
  \operatorname{Sat}^s_{\mathrm{loc}}
  \bigl(\Wu_{\mathrm{loc}}(y)\bigr).
\]
Thus $y'\in\mathcal I(y)$.
\end{proof}

\begin{lemma}\label{lem:codim-one-joint-integrability}
Assume that $\dim E^s=1$.  If there are non-centered times arbitrarily
close to zero, then $\Ws$ and $\Wu$ are jointly integrable.
\end{lemma}

\begin{proof}
Use the notation from the proof of \Cref{lem:off-diagonal-su}, and let
$J_n\subset\Ws_{\mathrm{loc}}(y_n)$ be the stable interval with endpoints
$y_n$ and $y_n'$.  These intervals converge to the stable interval
$J\subset\Ws_{\mathrm{loc}}(y)$ with endpoints $y$ and $y'$.

Fix $y''\in J$ and choose $y_n''\in J_n$ with $y_n''\to y''$.  Set
$I_n=[\min\{0,t_n\},\max\{0,t_n\}]$.  The map
\[
  s\longmapsto X^{-T_n}\bigl(p(y_0,s)\bigr),
  \qquad s\in I_n,
\]
is a continuous path in the one-dimensional stable leaf joining $y_n$ to
$y_n'$.  Hence there is $s_n$ between $0$ and $t_n$ such that
\[
  X^{-T_n}\bigl(p(y_0,s_n)\bigr)=y_n''.
\]
In particular, $s_n\to0$.  Repeating the final plaque-convergence argument
of \Cref{lem:off-diagonal-su}, with $s_n$ in place of $t_n$ and the same
$T_n$, gives
\[
  \Wu_\rho(y'')
  \subset
  \operatorname{Sat}^s_{\mathrm{loc}}
  \bigl(\Wu_{\mathrm{loc}}(y)\bigr)
\]
after shrinking $\rho$ if necessary.  Thus $J\subset\mathcal I(y)$.

{
The stable displacements corresponding to sufficiently small positive and
negative times lie on opposite sides of $y_0$.  Indeed, otherwise
continuity would give arbitrarily small $a,b>0$ such that
$p(y_0,-a)=p(y_0,b)$.  Then $\widehat Y^{-a}(y_0)$ and
$\widehat Y^b(y_0)$ would lie in the same weak unstable leaf, so
\[
 p\bigl(\widehat Y^{-a}(y_0),a+b\bigr)=\widehat Y^{-a}(y_0).
\]
By \Cref{lem:centered-times}, $a+b$ would be centered, contradicting
\Cref{lem:centered-accumulation} and the hypothesis.

Using the same $T_n$ as above, backward expansion along $\Ws$ and the
intermediate value theorem give $\bar t_n>0$ such that $\bar t_n\to0$ and
\[
 \ds\bigl(X^{-T_n}(y_0),X^{-T_n}(p(y_0,-\bar t_n))\bigr)=L,
\]
where $X^{-T_n}(p(y_0,-\bar t_n))$ lies on the side opposite to $y_n'$.
After passing to a subsequence,
\[
 X^{-T_n}(p(y_0,-\bar t_n))\longrightarrow\bar y'
\]
for some $\bar y'$ on the side of $y$ opposite to $y'$.  Repeating the
argument above with $-\bar t_n$ in place of $t_n$, and with the same $T_n$,
shows that the stable interval $\bar J$ with endpoints $y$ and $\bar y'$
is contained in $\mathcal I(y)$.  Hence $y$ is an interior point of
$J\cup\bar J$, and the unstable plaques through a smaller interval about
$y$ form a local product disk saturated by both $\Ws$ and $\Wu$.

Let $U$ be the set of points admitting such a disk.  It is open and
$X^t$-invariant.  Since $y_0$ was arbitrary and
$y\in\overline{X^{\R}(y_0)}\cap U$, the closure of every orbit meets $U$.
The set $M\setminus U$ is closed and invariant, hence, it must be empty. 
Therefore $U=M$, and $\Ws$ and $\Wu$ are jointly integrable.
}
\end{proof}

\subsection{The subbundle step}

Recall that, by definition,
$\mathcal I(y)\subset\Ws_{\mathrm{loc}}(y)$, and if
$w\in\mathcal I(y)$, then a local strong unstable plaque through $w$ is
contained in the local stable saturation of a strong unstable plaque
through $y$.

\begin{lemma}
\label{lem:general-subbundle}
Let $X^t$ be an Anosov flow whose strong unstable foliation is minimal,
and assume that $E^u$ is $C^1$.  If
$\mathcal I(y)\setminus\{y\}\neq\varnothing$ for some $y$, then there is a
non-zero continuous $DX^t$-invariant subbundle $E\subset E^s$.
The subbundle $E$ integrates to an $X^t$-invariant foliation
$\mathcal F\subset\Ws$.  Moreover,
$E\oplus E^u$ integrates to an $X^t$-invariant foliation
subfoliated by $\mathcal F$ and $\Wu$.
\end{lemma}

\begin{proof}
The subbundle $E$ was already constructed in Section~3,
\Cref{prop:holonomy-joint-foliation}; it remains only to prove that
$E=E_{\min}$ is non-zero.  Suppose, to the contrary, that $E$ is trivial,
and take
$y'\in\mathcal I(y)\setminus\{y\}$.  Thne there
is an $r>0$ such that
\[
 \operatorname{Hol}^u_{y,z}(y')\in\Ws_{\mathrm{loc}}(z)
 \qquad\text{for every }z\in\Wu_r(y).
\]
Choose $t_n\to+\infty$ such that $x_n:=X^{t_n}(y)\longrightarrow x$, $n\to\infty$. Since $y'\in\Ws_{\mathrm{loc}}(y)$, we also have $x_n':=X^{t_n}(y')\longrightarrow x$.

By \Cref{lem:holonomy-extremal-equality}, we can pick a small
$\rho>0$ there are finitely many points
$z_1,\ldots,z_k\in\Wu_\rho(x)$, including $z_1=x$, such that
\[
 \bigcap_{i=1}^k
 \bigl(D_x\operatorname{Hol}^u_{x,z_i}\bigr)^{-1}E^s(z_i)
 =E(x)=\{0\}.
\]
Let $\mathcal G_i$ be the pullback of the strong stable foliation by the
corresponding unstable holonomy.  The product of local submersions
defining the $\mathcal G_i$ has injective differential at $x$.
Consequently, after restricting to a local product box about $x$, the
common plaques of the continued foliations $\mathcal G_i$ contain no two
distinct points sufficiently close to $x$.  By the $C^1$ dependence of
unstable holonomy, the same conclusion holds, uniformly, on every weak
stable plaque obtained from $\Wos_{\mathrm{loc}}(x)$ by a sufficiently
short unstable holonomy, using the same points $z_1,\ldots,z_k$.

For all sufficiently large $n$, let $q_n\in\Wu_{\mathrm{loc}}(x)$ be the
point such that $x_n\in\Wos_{\mathrm{loc}}(q_n)$.  Then $q_n\to x$, and
both $x_n$ and $x_n'$ lie in $\Wos_{\mathrm{loc}}(q_n)$ and are
sufficiently close to $x$.  For each $i$, the point
\[
 \operatorname{Hol}^u_{q_n,z_i}(x_n)\in\Wu(x_n)
\]
is at uniformly bounded unstable distance from $x_n$.  Hence backward
contraction gives
\[
 X^{-t_n}\bigl(\operatorname{Hol}^u_{q_n,z_i}(x_n)\bigr)\in\Wu_r(y)
 \qquad\text{for all sufficiently large }n.
\]
Equivariance of unstable holonomy under the flow therefore implies
\[
 \operatorname{Hol}^u_{q_n,z_i}(x_n')
 \in\Ws_{\mathrm{loc}}\bigl(\operatorname{Hol}^u_{q_n,z_i}(x_n)\bigr),
 \qquad i=1,\ldots,k.
\]
Thus $x_n$ and $x_n'$ lie in the same plaque of every continued
$\mathcal G_i$.  Hence $x_n=x_n'$ for
all sufficiently large $n$, a contradiction.  Hence $E$ is non-trivial.
\end{proof}

%\begin{remark}
%The argument of \Cref{lem:codim-one-joint-integrability} gives an alternative
%proof (in any stable dimension), that the equivalence class $\mathcal I(y)$
 %is not discrete, even when $E^u$ is
%not $C^1$.
%\end{remark}

\subsection{The weak unstable preserving alternative}

The following proposition is the remaining ingredient in the proof of \Cref{thm:general-anosov}.
\begin{proposition}\label{prop:general-centered}
If every sufficiently small time is centered, then there exist $\lambda>0$
and a homeomorphism $\psi\colon N\to M$ such that
\[
  \psi\circ Y^t=X^{\lambda t}\circ\psi,
  \qquad t\in\R.
\]
\end{proposition}

\begin{proof}
For every sufficiently small $t$, centeredness gives a unique time
$\tau(y,t)$ such that
\[
  X^{-\tau(y,t)}\bigl(\widehat Y^t(y)\bigr)\in\Wu_X(y).
\]
The function $\tau(\cdot,t)$ is constant on strong unstable leaves.  By
minimality it is constant on $M$, and we can write it as $\tau(t)$.  Composition
of the local leaf relations gives
\[
  \tau(t+s)=\tau(t)+\tau(s)
\]
whenever $s,t$, and $s+t$ are sufficiently small.  Therefore
$\tau(t)=\lambda t$ for small $t$ and some $\lambda\in\R$.  Given arbitrary
$t$, choose $n$ so large that $t/n$ is in this range and iterate the local
leaf relation $n$ times.  This gives
\begin{equation}\label{eq:general-center-time}
  \varphi\bigl(\Wu_Y(Y^t(x))\bigr)
  =X^{\lambda t}\bigl(\Wu_X(\varphi(x))\bigr).
\end{equation}

Since $Y^t$ expands $\Wu_Y$ it is fairly clear from the above relation that $X^{\lambda t}$ should also expand $\Wu_X$ and, hence, $\lambda$ must be positive. Nevertheless, let us give a detailed argument for this; at the same time we will prepare for building the conjugacy. Define
\[
  \Theta_T=Y^{-T}\circ\varphi^{-1}\circ X^{\lambda T},
  \qquad T\geq0.
\]
For $z\in M$ and $0\leq s\leq1$, \eqref{eq:general-center-time} gives
\[
 Y^{-s}\bigl(\varphi^{-1}(X^{\lambda s}(z))\bigr)
 \in\Wu_Y\bigl(\varphi^{-1}(z)\bigr).
\]
The intrinsic unstable distance between these two points is uniformly
bounded by compactness.  Taking $z=X^{\lambda T}(y)$ and applying
backward contraction along $\Wu_Y$ therefore yields
\[
  d_Y^u\bigl(\Theta_{T+s}(y),\Theta_T(y)\bigr)
  \leq C e^{-\chi T},
  \qquad y\in M,\ 0\leq s\leq1.
\]
Hence $\Theta_T$ converges uniformly to a continuous onto map
$\Theta\colon M\to N$ satisfying
\begin{equation}\label{eq:general-factor}
  \Theta\circ X^{\lambda t}=Y^t\circ\Theta.
\end{equation}
If $\lambda=0$, then \eqref{eq:general-factor} and surjectivity imply that
$Y^t$ is the identity flow, which is impossible.

If $\lambda<0$ and $y'\in\Wu_X(y)$, then
\[
  d_X^u(X^{\lambda T}(y'),X^{\lambda T}(y))\longrightarrow0,\,\,T\to\infty.
\]
The images of corresponding points under $\varphi^{-1}$ lie on the same
$Y$-unstable leaf and their intrinsic distance also tends to zero.  Applying
$Y^{-T}$ only further contracts this distance.  Hence
$\Theta(y')=\Theta(y)$.  Minimality of $\Wu_X$ would make $\Theta$ constant,
contradicting surjectivity.  Thus $\lambda>0$.

Now let
\[
  \Phi_T=X^{-\lambda T}\circ\varphi\circ Y^T
  =\Theta_T^{-1}.
\]
Since $\lambda>0$, the same convergence argument, now using backward
contraction along $\Wu_X$, shows that $\Phi_T$ converges uniformly to a
continuous map $\psi\colon N\to M$.  Since $\Phi_T$ and $\Theta_T$ are mutual
inverses and both converge uniformly, their limits are inverse:
\[
  \psi=\Theta^{-1}.
\]
Hence $\psi$ is a homeomorphism.
Equation \eqref{eq:general-factor} gives: $\psi\circ Y^t=X^{\lambda t}\circ\psi$.
\end{proof}

\subsection{Proofs of Theorems~\ref{thm:general-anosov} and \ref{thm:codimension-one}, \Cref{add:nondegenerate}, and \Cref{cor:generic-conservative}}

We can now summarize the argument.  If every sufficiently small time is
centered, then \Cref{prop:general-centered} gives the first alternative.
Otherwise there are non-centered times arbitrarily close to zero, so
\Cref{lem:off-diagonal-su} gives 
$y'\in\mathcal I(y)\setminus\{y\}$, $y'\neq y$, and \Cref{lem:general-subbundle} gives
the second alternative.  This proves \Cref{thm:general-anosov}.

Suppose now that $E^u_X$ is not necessarily $C^1$, but that $X^t$ is locally
$su$-non-degenerate.  Then \Cref{lem:off-diagonal-su} rules out non-centered
times arbitrarily close to zero.  Hence every sufficiently small time is
centered, and \Cref{prop:general-centered} gives the conjugacy.  This proves
\Cref{add:nondegenerate}.

Now we prove \Cref{thm:codimension-one}. Suppose first that $\Ws_X$ and $\Wu_X$ are not jointly integrable.  Then the
preceding argument applies again.  If there are non-centered times arbitrarily
close to zero, \Cref{lem:codim-one-joint-integrability} gives joint
integrability, a contradiction.  Hence every sufficiently small time is
centered, and \Cref{prop:general-centered} gives the desired conjugacy.

It remains to consider the jointly integrable case.  By
\Cref{thm:plante-codimension-one}, $X^t$ is the suspension of a
codimension-one Anosov diffeomorphism
$f\colon\mathbb T_X\to\mathbb T_X$ with constant roof $C_X$. By Newhouse's theorem~\cite{Newhouse}, $\mathbb T_X$ is a torus.
The constant-height tori of the suspension are precisely the closures
of the strong unstable leaves of $X^t$.  Then
$\mathbb T_Y:=\varphi^{-1}(\mathbb T_X)$ is the closure of a strong unstable
leaf of $Y^t$.  In
particular, $\Wu_Y$ is not minimal, and by
\Cref{thm:anosov-alternative} we have that $Y^t$ is also a suspension of a
codimension-one Anosov diffeomorphism
$g\colon\mathbb T_Y\to\mathbb T_Y$ with constant roof $C_Y$.

Just as in \cite[Section~3.2]{Burniol}, one can see that $\varphi$
intertwines monodromies on $\mathbb T_Y$ and $\mathbb T_X$:
\[
  \varphi_*\circ g_*=f_*\circ \varphi_*.
\]
Franks' classification~\cite{Franks} therefore gives a homeomorphism
$h\colon\mathbb T_Y\to\mathbb T_X$ such that
$
  h\circ g=f\circ h.
$
Suspending $h$ and taking $\lambda=C_X/C_Y$ gives
$
  \psi\circ Y^t=X^{\lambda t}\circ\psi
$
with $\lambda>0$.  This proves \Cref{thm:codimension-one}.

We finish by proving \Cref{cor:generic-conservative}.
Recall that if $X^t\in\mathcal A(M)$, then $E^u_X$ is $C^1$; this follows by
applying \cite[Section~2.1]{GogolevRodriguezHertz} to the reversed flow.
Hence \Cref{thm:general-anosov} applies to such flows.  Further,
\cite[Proposition~4.1]{GogolevRodriguezHertz}, again applied after reversing
time, gives a $C^1$-open and
$C^\infty$-dense subset $\mathcal V\subset\mathcal A(M)$ such that, for
$X^t\in\mathcal V$, there is no non-trivial continuous
$DX^t$-invariant subbundle $E\subset E^s_X$ which integrates to a foliation
$\mathcal F\subset\Ws_X$ and for which $E\oplus E^u_X$ integrates to a
foliation subfoliated by $\mathcal F$ and $\Wu_X$.

\Cref{thm:general-anosov} applies to $X^t$ and $Y^t$, and the second
alternative is ruled out by the property of $\mathcal V$.  Hence we obtain
$\lambda>0$ and a conjugacy $\psi\colon M\to M$ such that
\[
  \psi\circ Y^t=X^{\lambda t}\circ\psi.
\]
Finally, \cite[Theorem~1.2]{GogolevRodriguezHertz} yields $C^1$ regularity of the conjugacy $\psi$.

\section{Proof of
\texorpdfstring{\Cref{thm:compact-extension}}{Theorem 5}}
\label{sec:compact-extensions}

We now prove \Cref{thm:compact-extension}.  Throughout this section,
$\pi\colon M\to B$ is a principal $K$-bundle and $X^t$ is an isometric
extension of the Anosov flow $\bar X^t$.

Fix an unstable equivalence $\varphi\colon N\to M$.  As in the previous proof, transplant $Y^t$ to $M$:
\[
  \widehat Y^t=\varphi\circ Y^t\circ\varphi^{-1}.
\]
It is a continuous flow preserving the strong unstable foliation of $X^t$.

\subsection{Centering the transplanted flow}

Fix $\delta_1>0$ so that the following local product intersection is defined
for $|t|<\delta_1$, and let
\begin{equation}
  p(y,t)=[y,\widehat Y^t(y)]
  =\Ws_{\mathrm{loc}}(y)
   \cap\Wcu_{\mathrm{loc}}\bigl(\widehat Y^t(y)\bigr).
\end{equation}
Call $t$ \emph{centered} if $p(y,t)=y$ for every $y\in M$.
As in \Cref{lem:centered-times}, the set of $y$ satisfying $p(y,t)=y$ is
closed and $\Wu$-saturated.  By minimality assumption if $p(y,t)=y$ for some $y\in M$ then $t$ is centered.

For $(s,k)\in\R\times K$ let
\begin{equation}
  \Phi_{s,k}=R_k\circ X^s.
\end{equation}
For small $s$ the maps $\Phi_{s,k}$ give local charts on the center leaves.

\begin{lemma}\label{lem:group-centering}
Every sufficiently small time is centered.
\end{lemma}

\begin{proof}
This is almost the same as \Cref{lem:off-diagonal-su}; we include a brief
proof for completeness.  Suppose otherwise and choose non-centered times
with $0<|t_n|\to0$.  Local product
structure inside $\Wcu$ gives unique small $\sigma(y,t)\in\R$ and
$\kappa(y,t)\in K$ such that
\[
  q(y,t):=\Phi_{\sigma(y,t),\kappa(y,t)}(p(y,t))
  \in\Wu(\widehat Y^t(y)).
\]
These coordinates tend uniformly to $(0,\mathrm{id}_K)$ as $t_n\to0$,
$\du(q(y,t_n),\widehat Y^{t_n}(y))$ is uniformly bounded, and, for fixed $t_n$,
$q(\,\mathord{\cdot}\,,t_n)$ restricts to a center-stable holonomy
homeomorphism on each unstable leaf.

Fix $y_0$ and let $p_n=p(y_0,t_n)\neq y_0$.  As in
\Cref{lem:off-diagonal-su}, choose $T_n\to\infty$ so that, after passing to
a subsequence,
\[
  X^{-T_n}(y_0)\longrightarrow y,
  \qquad
  X^{-T_n}(p_n)\longrightarrow
  y'\in\Ws_{\mathrm{loc}}(y)\setminus\{y\}.
\]
Commutation of $X^t$ with the fiber action gives
\[
  \sup_z d\bigl(X^{-T_n}(q(z,t_n)),X^{-T_n}(p(z,t_n))\bigr)
  \longrightarrow0,
\]
while the uniform unstable bound and backward contraction give
\[
  \sup_z \du\bigl(X^{-T_n}(q(z,t_n)),
  X^{-T_n}(\widehat Y^{t_n}(z))\bigr)\longrightarrow0;
\]
here the suprema are taken over $z\in\Wu(y_0)$.  The plaque-limit argument from
\Cref{lem:off-diagonal-su} therefore produces a non-trivial local
$su$-integrability relation.  This contradicts local $su$-non-degeneracy, so
every sufficiently small time is centered.
\end{proof}

\subsection{The center homomorphism}

By \Cref{lem:group-centering}, the path $t\mapsto \widehat Y^t(y)$ lies in $\Wcu_{\mathrm{loc}}(y)$. Hence there is a unique pair
\[
  (\tau(y,t),k(y,t))\in\R\times K
\]
near $(0,\mathrm{id}_K)$ such that
\begin{equation}
  \widehat Y^t(y)\in
  \Wu\bigl(\Phi_{\tau(y,t),k(y,t)}(y)\bigr).
\end{equation}
Because $\widehat Y^t$, $X^t$, and the isometric action all preserve the
strong unstable foliation, local uniqueness implies that $\tau(y,t)$ and
$k(y,t)$ are constant along strong unstable leaves.  Minimality of $\Wu$ implies that they are independent of $y$; we write them as
$\tau(t)$ and $k(t)$.

 Local uniqueness of the center coordinate gives
\[
  \tau(t+s)=\tau(t)+\tau(s),
  \qquad
 k(s)k(t)=k(t+s)=k(t)k(s)
\]
whenever $s,t$, and $s+t$ are sufficiently small.  Consequently
\[
  \tau(t)=\lambda t
\]
for small $t$ and some $\lambda\in\R$, and the local map $k$ extends uniquely
to a continuous homomorphism $k\colon\R\to K$.  Define the associated
rotation flow by $R^t(y)=y\cdot k(t)$.  Iterating the local leaf relation
shows that, for every $t\in\R$,
\begin{equation}\label{eq:group-leaf-relation}
  \varphi\bigl(\Wu_Y(Y^t(x))\bigr)=\Wu(Z^t(\varphi(x))),
\end{equation}
where
\begin{equation}\label{eq:twisted-target-flow}
  Z^t:=R^t\circ X^{\lambda t}.
\end{equation}

The proof that $\lambda>0$ is the same as in
\Cref{prop:general-centered}, with very minor adjustments.

\subsection{The correcting maps}

With $Z^t$ as in \eqref{eq:twisted-target-flow},  $\lambda>0$, define
the correcting maps
\begin{equation}
  \varphi_T=Z^{-T}\circ\varphi\circ Y^T,
  \qquad T\geq0,
\end{equation}
for which \eqref{eq:group-leaf-relation} gives
$\varphi_T(x)\in\Wu(\varphi(x))$.  As in the proof of
\Cref{prop:general-centered}---now using $Z^{-T}$, which also contracts $\Wu$, ---the maps
$\varphi_T$ converge uniformly to a continuous $\psi\colon N\to M$ such that
\begin{equation}
  \psi(x)\in\Wu(\varphi(x)),
  \qquad
  \sup_x \du(\psi(x),\varphi(x))<\infty,
\end{equation}
and, passing to the limit in $\varphi_T\circ Y^t=Z^t\circ\varphi_{T+t}$, we have
\begin{equation}
  \psi\circ Y^t
  =Z^t\circ\psi
  =R^t\circ X^{\lambda t}\circ\psi.
\end{equation}

Finally, set $\Theta_T=Y^{-T}\circ\varphi^{-1}\circ Z^T$.  The same argument
shows that $\Theta_T$ converges uniformly to a continuous onto map $\Theta$.
Since $\varphi_T=\Theta_T^{-1}$ and both families converge uniformly, the
limits are mutually inverse; hence $\psi$ is a homeomorphism.  This finishes
the proof of \Cref{thm:compact-extension}.

\section{Geometric applications}
\label{sec:geometric-applications}

The corollaries follow by combining the dynamical results above with the
standard marked-length-spectrum rigidity theorems.  We give the short
reductions to these well-known results.

\subsection{Geodesic flows}

\begin{proof}[Proof of \Cref{cor:geodesic-rigidity}]
The geodesic flows are contact Anosov flows, so
\Cref{prop:contact-su-nondegenerate} and~\Cref{add:nondegenerate} give $\lambda>0$ and a
homeomorphism $\psi\colon T^1N\to T^1M$ satisfying
$\psi\circ Y^t=X^{\lambda t}\circ\psi$.  Since $n\geq3$, this conjugacy
induces an isomorphism $\rho\colon\pi_1(N)\to\pi_1(M)$ and hence a homotopy
equivalence.  Comparison of periodic orbits gives
\[
  \ell_{g_0}(\rho[\gamma])
  =\lambda\,\ell_g([\gamma])
  =\ell_{\lambda^2g}([\gamma]),
  \qquad [\gamma]\in[\pi_1(N)].
\]
Thus $(N,\lambda^2g)$ and $(M,g_0)$ have the same marked length spectrum.
Marked length spectrum rigidity for negatively curved locally symmetric
metrics \cite{Hamenstadt,BCG} gives the asserted isometry.
\end{proof}

\begin{proof}[Proof of \Cref{cor:geodesic-local-rigidity}]
Again \Cref{prop:contact-su-nondegenerate} and~\Cref{add:nondegenerate} give
$\lambda>0$ and a conjugacy
$\psi\circ Y^t=X^{\lambda t}\circ\psi$.  The correcting maps in the proof
of \Cref{add:nondegenerate} show that $\psi$ is homotopic to the given
unstable equivalence, hence to the identity.  It therefore preserves the
marking, and comparison of periods gives again
\[
  \ell_{g_0}([\gamma])
  =\lambda\,\ell_g([\gamma])
  =\ell_{\lambda^2g}([\gamma]),
  \qquad [\gamma]\in[\pi_1(M)].
\]
For any fixed non-trivial $[\gamma_0]$, the same identity shows that
$\lambda=\ell_{g_0}([\gamma_0])/\ell_g([\gamma_0])$, so $\lambda\to1$ as
$g\to g_0$.  After decreasing $\varepsilon$, local marked-length-spectrum
rigidity \cite[Corollary~1.1]{GuillarmouLefeuvre} applies to
$\lambda^2g$ and $g_0$ and gives the asserted isometry.
\end{proof}

\subsection{Frame flows}

\begin{proof}[Proof of \Cref{rem:frame-mls}]
Each oriented frame flow is an isometric extension on the principal bundle
$\operatorname{Fr}(M_i,g_i)\to T^1M_i$ with fiber $SO(n-1)$; local
$su$-non-degeneracy follows from \Cref{prop:contact-su-nondegenerate}.
Write $R_k^X$ and $R_k^Y$ for the
right $SO(n-1)$-actions on $\operatorname{Fr}(M_1,g_1)$ and $\operatorname{Fr}(M_2,g_2)$, respectively.
The actions $R^X$ and $R^Y$ commute with $X^t$ and $Y^t$, respectively.  Thus
\Cref{thm:compact-extension}
gives $\lambda>0$, a homomorphism $\kappa\colon\R\to SO(n-1)$, and a homeomorphism
$\psi\colon \operatorname{Fr}(M_2,g_2)\to \operatorname{Fr}(M_1,g_1)$ such that
\[
  \psi\circ Y^t=Z^t\circ\psi,
  \qquad Z^t=R_{\kappa(t)}^X\circ X^{\lambda t}.
\]

We only need to check that $\psi$ preserves the compact fibers.  For
$k\in SO(n-1)$, define
\[
  Q_k:=\psi\circ R_{k^{-1}}^Y\circ\psi^{-1}.
\]
Since $R_{k^{-1}}^Y$ commutes with $Y^t$, the map $Q_k$ commutes with $Z^t$.
When $k$ is close to the identity element, the map $Q_k$ is uniformly close to the
identity map, and hyperbolicity implies that
$Q_k(x)\in W^c_Z(x)=W^c_X(x)$.  Hence there are continuous functions
$a_k(x)\in SO(n-1)$ and $s_k(x)\in\R$ such that
\[
  Q_k(x)=R_{a_k(x)}^X\bigl(X^{s_k(x)}(x)\bigr).
\]
Uniqueness of these center coordinates and preservation of strong unstable
leaves imply that $a_k$ and $s_k$ are constant along $\Wu_X$-leaves.
Minimality makes them independent of $x$; write the resulting constants as
$a(k)$ and $s(k)$.

We have $Q_{kh}=Q_k\circ Q_{h}$ for all $k,h\in SO(n-1)$. Finite-order elements are
dense in $SO(n-1)$.  Pick such an element $k$ close to the identity such that $k^N=\mathrm{id}$ for some $N\geq 1$.  Then
\[
  \mathrm{id}=Q_k^N
  =R_{a(k)^N}^X\circ X^{N s(k)}.
\]
Projecting to the base gives $\bar X^{N s(k)}=\mathrm{id}$, and hence
$s(k)=0$.  Such $k$ are dense in the neighborhood of $\mathrm{id}$; hence, $s$ vanishes near the identity.
Thus $Q_k$ is vertical for $k$ near the identity.  Since $SO(n-1)$ is
connected, every element is a finite product of elements in this
neighborhood, and the group law shows that $Q_k$ is vertical for every
$k\in SO(n-1)$.

For $y\in \operatorname{Fr}(M_2,g_2)$, its fiber is its $SO(n-1)$-orbit, and
\[
  \psi\bigl(R_k^Y(y)\bigr)=Q_{k^{-1}}\bigl(\psi(y)\bigr),
  \qquad k\in SO(n-1).
\]
Since $Q_{k^{-1}}$ is vertical, $\psi$ maps the fiber through $y$ into the
fiber through $\psi(y)$.  Its restriction to this fiber is a continuous
injection between compact connected manifolds of the same dimension.  By
invariance of domain and compactness, $\psi$ maps fibers onto fibers.
Consequently $\psi$ descends to a conjugacy of the base geodesic flows up to
the constant time change.
\end{proof}

\bigskip
\noindent
Department of Mathematics, The Ohio State University, Columbus, OH 43210, USA

\noindent
\textit{Email address:} \texttt{gogolyev.1@osu.edu}


\begin{thebibliography}{FMP26}

\bibitem[A95]{Abe}
R. Abe,
\emph{Geometric approach to rigidity of horocycles},
Tokyo J. Math. 18 (1995), no.~2, 271--283.

\bibitem[A67]{Anosov}
D. V. Anosov,
\emph{Geodesic flows on closed Riemannian manifolds of negative curvature},
Trudy Mat. Inst. Steklov. 90 (1967), 3--210; English translation,
Proc. Steklov Inst. Math. 90 (1967), 1--235.

%\bibitem[B95]{Barbot}
%T. Barbot,
%\emph{Caract\'erisation des flots d'Anosov en dimension 3 par leurs
%feuilletages faibles},
%Ergodic Theory Dynam. Systems 15 (1995), no.~2, 247--270.

\bibitem[BCG95]{BCG}
G. Besson, G. Courtois and S. Gallot,
\emph{Entropies et rigidit\'es des espaces localement sym\'etriques de
courbure strictement n\'egative},
Geom. Funct. Anal. 5 (1995), no.~5, 731--799.

\bibitem[B75]{Brin}
M. I. Brin,
\emph{Topological transitivity of one class of dynamical systems and flows
of frames on manifolds of negative curvature},
Funct. Anal. Appl. 9 (1975), no.~1, 8--16.

\bibitem[BC25]{Burniol}
S. Burniol Clotet,
\emph{Rigidity of the unstable foliation},
arXiv:2511.22738, 2025.

\bibitem[BG80]{BrinGromov}
M. Brin and M. Gromov,
\emph{On the ergodicity of frame flows},
Invent. Math. 60 (1980), no.~1, 1--7.

\bibitem[FMP26]{FMP}
S. Fenley, K. Mann and R. Potrie,
\emph{Exotic codimension one Anosov flows},
preprint, arXiv:2605.25082, 2026.

\bibitem[F87]{Flaminio}
L. Flaminio,
\emph{An extension of Ratner's rigidity theorem to $n$-dimensional
hyperbolic space},
Ergodic Theory Dynam. Systems 7 (1987), no.~1, 73--92.

\bibitem[FS90]{FlaminioSpatzier}
L. Flaminio and R. J. Spatzier,
\emph{Geometrically finite groups, Patterson--Sullivan measures and
Ratner's rigidity theorem},
Invent. Math. 99 (1990), no.~3, 601--626.

\bibitem[F70]{Franks}
J. Franks,
\emph{Anosov diffeomorphisms},
Global Analysis (Proc. Sympos. Pure Math., Vol.~XIV, Berkeley, Calif., 1968),
Amer. Math. Soc., Providence, R.I., 1970, 61--93.

\bibitem[GRH23]{GogolevRodriguezHertz}
A. Gogolev and F. Rodriguez Hertz,
\emph{Smooth rigidity for codimension one Anosov flows},
Proc. Amer. Math. Soc. 151 (2023), no.~7, 2975--2988.

\bibitem[GRH25]{GogolevRodriguezHertz25}
A. Gogolev and F. Rodriguez Hertz,
\emph{Smooth rigidity for very non-algebraic Anosov diffeomorphisms of
codimension one},
Israel J. Math. 269 (2025), no.~2, 801--852.

\bibitem[GL19]{GuillarmouLefeuvre}
C. Guillarmou and T. Lefeuvre,
\emph{The marked length spectrum of Anosov manifolds},
Ann. of Math. (2) 190 (2019), no.~1, 321--344.

\bibitem[H99]{Hamenstadt}
U. Hamenst\"adt,
\emph{Cocycles, symplectic structures and intersection},
Geom. Funct. Anal. 9 (1999), no.~1, 90--140.

\bibitem[L04]{Liverani}
C. Liverani,
\emph{On contact Anosov flows},
Ann. of Math. (2) 159 (2004), no.~3, 1275--1312.

\bibitem[M83]{Marcus}
B. Marcus,
\emph{Topological conjugacy of horocycle flows},
Amer. J. Math. 105 (1983), no.~3, 623--632.

%\bibitem[L23]{Lefeuvre}
%T. Lefeuvre,
%\emph{Isometric extensions of Anosov flows via microlocal analysis},
%Commun. Math. Phys. 399 (2023), no.~1, 453--479.

\bibitem[N70]{Newhouse}
S. E. Newhouse,
\emph{On codimension one Anosov diffeomorphisms},
Amer. J. Math. 92 (1970), 761--770.

\bibitem[P72]{Plante}
J. F. Plante,
\emph{Anosov flows},
Amer. J. Math. 94 (1972), 729--754.

\bibitem[R82]{Ratner}
M. Ratner,
\emph{Rigidity of horocycle flows},
Ann. of Math. (2) 115 (1982), no.~3, 597--614.

\end{thebibliography}
\end{document}